\documentclass[12pt]{article}

\usepackage{enumerate}
\usepackage{amsmath}
\usepackage{amssymb,latexsym}
\usepackage{amsthm}
\usepackage{color}

\usepackage{graphicx}

\usepackage{hyperref}

\usepackage{amscd}

\newtheorem{theorem}{Theorem}[section]

\newtheorem{lemma}[theorem]{Lemma}
\newtheorem{corollary}[theorem]{Corollary}

\newtheorem{proposition}[theorem]{Proposition}

\newtheorem{example}[theorem]{Example}

\newtheorem{remark}[theorem]{Remark}

\newcommand{\cL}{{\mathcal L}}

\newcommand{\F}{{\mathbb F}}

\newcommand{\la}{\langle}
\newcommand{\ra}{\rangle}

\newcommand{\PG}{\mathrm{PG}}
\newcommand{\N}{\mathrm{N}}
\newcommand{\tr}{\mathrm{Tr}}

\title{On the number of roots of some linearized polynomials}
\author{Olga Polverino and Ferdinando Zullo\thanks{
This research was supported by the project ``VALERE: VAnviteLli pEr la RicErca" of the University of Campania ``Luigi Vanvitelli'', and by the Italian National Group for Algebraic and Geometric Structures and their Applications (GNSAGA
- INdAM). }}

\begin{document}
\maketitle

\begin{abstract}
Linearized polynomials appear in many different contexts, such as rank metric codes, cryptography and linear sets, and the main issue regards the characterization of the number of roots from their coefficients.
Results of this type have been already proved in \cite{Cs2019,teoremone,McGuireSheekey}.
In this paper we provide bounds and characterizations on the number of roots of linearized polynomials of this form
\[ ax+b_0x^{q^s}+b_1x^{q^{s+n}}+b_2x^{q^{s+2n}}+\ldots+b_{t-1}x^{q^{s+n(t-1)}} \in \F_{q^{nt}}[x], \]
with $\gcd(s,n)=1$.
Also, we characterize the number of roots of such polynomials directly from their coefficients, dealing with matrices which are much smaller than the relative Dickson matrices and the companion matrices used in the previous papers.
Furthermore, we develop a method to find explicitly the roots of a such polynomial by finding the roots of a $q^n$-polynomial.
Finally, as an applications of the above results, we present a family of linear sets of the projective line whose points have a small spectrum of possible weights, containing most of the known families of scattered linear sets. In particular, we carefully study the linear sets in $\PG(1,q^6)$ presented in \cite{CMPZ}.
\end{abstract}



\section{Introduction}

Linearized polynomials over $\F_{q^n}$ are important objects in the theory of finite fields and in finite geometry since they correspond to $\F_q$-linear transformations of the $n$-dimensional $\F_q$-vector space $\F_{q^n}$, and can be used to describe related objects such as $\F_q$-subspaces, rank metric codes, $\F_q$-linear sets.
A fundamental problem in the theory of linearized polynomials over finite fields is the characterization of the number of roots in the coefficient field directly from their coefficients. In this paper we provide results of this type.

\smallskip

A $\sigma$-\emph{polynomial} (or \emph{linearized polynomial}) over $\F_{q^n}$ is a polynomial of the form
\[f(x)=\sum_{i=0}^t a_i x^{\sigma^i},\]
where $a_i\in \F_{q^n}$, $t$ is a positive integer and $\sigma$ a generator of the Galois group $\mathrm{Gal}(\F_{q^n}\colon \F_q)$.
Furthermore, if $a_t \neq 0$ we say that $t$ is the $\sigma$-\emph{degree} of $f$.
We will denote by $\mathcal{L}_{n,q,\sigma}$ the set of all $\sigma$-polynomials over $\F_{q^n}$ (or simply by $\mathcal{L}_{n,q}$ if $x^\sigma=x^q$) and by $\tilde{\mathcal{L}}_{n,q,\sigma}$ (or by $\tilde{\mathcal{L}}_{n,q}$ if $x^\sigma=x^q$) the following quotient $\mathcal{L}_{n,q,\sigma}/(x^{\sigma^n}-x)$.
The polynomials in $\tilde{\mathcal{L}}_{n,q,\sigma}$ are precisely those which define $\F_q$-linear maps.
In the remainder of this paper we shall
always silently identify the elements of $\tilde \cL_{n,q,\sigma}$ with the
endomorphisms of $\F_{q^n}$ they represent and, as such,
speak also of \emph{kernel} and \emph{rank} of a polynomial.
Clearly, the kernel of $f\in \tilde \cL_{n,q,\sigma}$ coincides with the set of the roots of $f$ and as usual $\dim_{\F_q} \mathrm{Im}(f)+\dim_{\F_q} \ker(f)=n$.

\smallskip

The number of roots of a $\sigma$-polynomial over a cyclic extension of a field $\mathbb{F}$ (including the case of finite fields) is bounded as follows.

\begin{theorem}\cite[Theorem 5]{GQ2009}
\label{Gow}
Let $\mathbb{L}$ be a cyclic extension of a field $\F$ of degree $n$, and suppose that $\sigma$ generates the Galois group of $\mathbb{L}$ over $\F$. Let $k$ be an integer satisfying $1 \leq k \leq n$, and let $a_0,a_1,\ldots,a_k$ be elements of $\mathbb{L}$, not all of them are zero. Then the $\F$-linear transformation of $\mathbb{L}$ defined as
\[ f(x)=a_0x+a_1x^\sigma+\cdots+a_{k}x^{\sigma^{k}} \]
has kernel with dimension at most $k$ in $\mathbb{L}$.
\end{theorem}

In \cite{teoremone}, $\sigma$-polynomials over finite fields for which the dimension of the kernel coincides with their $\sigma$-degree are called \emph{linearized polynomials with maximum kernel}.
In order to determine the number of roots over $\F_{q^n}$ of a $\sigma$-polynomial we recall the following two matrices:
let $f(x)=a_0x+a_1x^{\sigma}+\ldots+a_{k}x^{\sigma^k}$ be a $\sigma$-polynomial over $\F_{q^n}$ with $\sigma$-degree $k$ with $1\leq k\leq n$, then its \emph{Dickson matrix}\footnote{This is sometimes called \emph{autocirculant matrix}.} $D(f)$ is defined as
\[D(f):=
\begin{pmatrix}
a_0 & a_1 & \ldots & a_{n-1} \\
a_{n-1}^{\sigma} & a_0^{\sigma} & \ldots & a_{n-2}^{\sigma} \\
\vdots & \vdots & \vdots & \vdots \\
a_1^{\sigma^{n-1}} & a_2^{\sigma^{n-1}} & \ldots & a_0^{\sigma^{n-1}}
\end{pmatrix} \in \F_{q^n}^{n\times n}
,\]
where $a_i=0$ for $i>k$, and its \emph{companion matrix} $C_f$ is defined as
\[C_f=\left(
	\begin{array}{cccccccccccc}
	0 & 0 & \cdots & 0 & -a_0/a_k \\
	1 & 0 & \cdots & 0 & -a_1/a_k \\
	0 & 1 & \cdots & 0 & -a_2/a_k \\
	\vdots & \vdots & & \vdots & \vdots \\
	0 & 0 & \cdots & 1 & -a_{k-1}/a_k
	\end{array}
	\right)\in \F_{q^n}^{k\times k}.\]

\smallskip

We briefly recall the roles of these matrices for the known results about the number of roots of a linearized polynomial.
It is well-known that for a $q$-polynomial $f$ over $\F_{q^n}$ we have that $\dim_{\F_q}\,\ker\,f=n-\mathrm{rk}\,D(f)$, see e.g. \cite[Proposition 4.4]{wl}.
Very recently, Csajb\'ok in \cite{Cs2019} shows that in order to determine the rank of $D(f)$ it is enough to look at some of its special minors.
Denote by $D_m(f)$ the $(n-m)\times(n-m)$ matrix obtained from $D(f)$ after removing its first $m$ columns and last $m$ rows.

\begin{theorem}\cite[Theorem 3.4]{Cs2019}\label{bence}
Let $f(x)=a_0x+a_1x^{\sigma}+\ldots+a_kx^{\sigma^k} \in \tilde{\mathcal{L}}_{n,q,\sigma}$.
Then $\dim_{\F_q} \ker\, f=m$ if and only if
\[ \det D_0(f)=\det D_1(f)= \ldots =\det D_{m-1}(f)=0 \]
and $\det D_m(f) \neq 0$.
\end{theorem}

In \cite{teoremone}, jointly with Csajb\'ok and Marino, we prove the following characterization of $\sigma$-polynomials with maximum kernel.

\begin{theorem}\cite[Theorem 1.2]{teoremone}\label{maxker}
Consider
\[f(x)=a_0x+a_1x^{\sigma}+\cdots+a_{k-1}x^{\sigma^{k-1}}-x^{\sigma^k},\]
Then $f(x)$ is of maximum kernel if and only if the matrix
\[C_f C_f^{\sigma} \cdot \ldots \cdot C_f^{\sigma^{n-1}}=I_k,\]
where $C_f$ is the companion matrix of $f$, $C_f^{\sigma^i}$ is the matrix obtained from $C_f$ by applying to each of its entries the automorphism $x\mapsto x^{\sigma^i}$ and $I_k$ is the identity matrix of order $k$.
\end{theorem}

McGuire and Sheekey in \cite{McGuireSheekey} generalize the previous result as follows.

\begin{theorem}\cite[Theorem 6]{McGuireSheekey}\label{rk}
Consider
\[f(x)=a_0x+a_1x^{\sigma}+\cdots+a_{k-1}x^{\sigma^{k-1}}+a_kx^{\sigma^k} \in \tilde{\mathcal{L}}_{n,q,\sigma}.\]
Then
	\[\dim_{\F_q} \ker\,f = n-\mathrm{rk}\, E_1,\]
where $E_1=C_f C_f^{\sigma} \cdots C_f^{\sigma^{n-1}}-I_k$.
\end{theorem}

\medskip

Our aim is to prove similar results for linearized polynomials of type
\begin{equation}\label{form}
ax+b_0x^{\sigma}+b_1x^{\sigma q^n}+b_2x^{\sigma q^{2n}}+\ldots+b_{t-1}x^{\sigma q^{n(t-1)}} \in {\mathcal{L}}_{nt,q}
\end{equation}
with $a \neq 0$ and $\sigma\colon x \mapsto x^{q^s}$ a generator of the Galois group $\mathrm{Gal}(\F_{q^n}\colon\F_q)$, where $x^{\sigma q^{ni}}:=x^{q^{s+ni}}$.
More precisely, our main results are the following.

\smallskip


\begin{theorem}\label{main}
Let
\[ f(x)=-x+b_0x^{\sigma}+b_1x^{\sigma q^n}+b_2x^{\sigma q^{2n}}+\ldots+b_{t-1}x^{\sigma q^{n(t-1)}} \in \mathcal{L}_{nt,q}, \]
where $\sigma \in \mathrm{Aut}(\F_{q^{nt}})$ such that $\sigma|_{\F_{q^n}}\colon\F_{q^n}\rightarrow\F_{q^n}$ has order $n$.
Let $G(x)$ be the $q^n$-polynomial such that $f(x)=(G\circ\sigma)(x)-x$, i.e. $G(x)=\sum_{i=0}^{t-1} b_ix^{q^{ni}}$.
Then
\begin{enumerate}
  \item $\dim_{\F_q} \ker\,f \leq t$;
  \item $\ker f=\{0\}$ if and only if $\ker((G\circ\sigma)^n-\mathrm{id})=\{0\}$(\footnote{We denote by $H^n$ the composition $H\circ H\circ \ldots \circ H$ $n$ times.}).\\
  \noindent More generally,
  \item $\dim_{\F_q}\ker f=\dim_{\F_{q^n}}\ker ((G\circ\sigma)^n-\mathrm{id})$.
\end{enumerate}
In particular, $\dim_{\F_q}\ker f=t$ if and only if $(G\circ \sigma)^n=\mathrm{id}$.
\end{theorem}

Using this result, in Section \ref{sec:matrix} we prove the following theorem that, similarly to Theorems \ref{maxker} and \ref{rk}, characterizes the number of roots of a linearized polynomial by giving relations on their coefficients and using a much smaller matrix than those used for the general case.

\begin{theorem}\label{main1}
Let
\[ f(x)=-x+b_0x^{\sigma}+b_1x^{\sigma q^n}+b_2x^{\sigma q^{2n}}+\ldots+b_{t-1}x^{\sigma q^{n(t-1)}} \in \mathcal{L}_{nt,q}, \]
where $\sigma$ is a generator of $\mathrm{Gal}(\F_{q^n}\colon\F_q)$. Then
$\dim_{\F_q}\ker f=h$ if and only if
\begin{equation}\label{matrixcondition}
\mathrm{rk}(D^{\tau^{n-1}} \cdot D^{\tau^{n-2}} \cdot \ldots \cdot D-J^s)=t-h,
\end{equation}
where $\tau=\sigma^{-1}$, $J:=\left(\begin{array}{cccccccccccc}
	0 & 1 & 0 & \cdots & 0 & 0 \\
    0 & 0 & 1 & \cdots & 0 & 0 \\
	\vdots & \vdots & & & \vdots & \vdots \\
	0 & 0 & 0 & \cdots & 0 & 1 \\
	1 & 0 & 0&  \cdots & 0 & 0
	\end{array}\right)\in \F_{q^n}^{t\times t}$,
\[ D=D_{red}(f):=\left(\begin{array}{cccccccccccc}
	b_0 & b_1 & \cdots & b_{t-1} \\
	b_{t-1}^{q^n} & b_0^{q^n} & \cdots & b_{t-2}^{q^n} \\
	\vdots & \vdots & & \vdots &  \\
	b_{1}^{q^{(t-1)n}} & b_{2}^{q^{(t-1)n}} & \cdots & b_0^{q^{n(t-1)}}
	\end{array}\right)\in \F_{q^n}^{t\times t} \]
and $s$ is the minimum integer such that $1 \leq s \leq nt$ and $\tau\colon x \mapsto x^{q^{s}}$.
\end{theorem}

We call $D_{red}(f)$ the \emph{restricted Dickson matrix} associated with $f$.
Note that $D_{red}(f)$ corresponds to the Dickson matrix of the $q^n$-polynomial $G$ defined in Theorem \ref{main} and it is a submatrix of $D(f)$ of order $t$.

In Section \ref{sec:findingroots} we show a method to find the roots of polynomials in these family and in Section \ref{sec:trin} we apply our results to trinomials also investigated in \cite{McGuireMueller}.
We also show explicit calculations for some fixed parameters in Section \ref{sec:criteria}.
This class of polynomials is quite large and contains properly linearized polynomials appearing in \cite{BZZ,CMPZ,CsMZ2018,LP2001,Sheekey2016,ZZ} and defining important examples of MRD-codes and scattered linear sets, see Section \ref{sec:application}. In particular, in Theorem \ref{thm:7.3} we completely determine the scattered linear sets appearing in \cite{CMPZ} when $n=6$.

\section{Proof of Theorem \ref{main}}

In this section we will investigate the number of roots of a $q$-polynomial of the following form
\[ f(x)=ax+b_0x^{\sigma}+b_1x^{\sigma q^n}+b_2x^{\sigma q^{2n}}+\ldots+b_{t-1}x^{\sigma q^{n(t-1)}} \in \F_{q^{nt}}[x] \]
with $a \neq 0$ and $\sigma \in \mathrm{Aut}(\F_{q^{nt}})$ such that $\sigma|_{\F_{q^n}}\colon\F_{q^n}\rightarrow\F_{q^n}$ has order $n$.
We may assume that $a=-1$, since $a \neq 0$ and $\ker f = \ker (-a^{-1} f)$.
Note that $f$ is neither a $q^n$-polynomial nor a $\sigma$-polynomial.

\smallskip

To prove our main theorem (Theorem \ref{main}), we will need the following result by Dempwolff, Fisher and Herman from \cite{DFH}, see also \cite[Theorem 2.2]{teoremone}.

\begin{theorem}\label{fixsemi}
Let $T$ be an invertible semilinear transformation of $\mathbb{V}=V(t,q^m)$ of order $m$, with companion automorphism $\tau\in \mathrm{Gal}(\F_{q^m}\colon\F_q)$.
Then $\mathrm{Fix}(T)=\{\mathbf{v} \in \mathbb{V} \colon T(\mathbf{v})=\mathbf{v}\}$ is a $t$-dimensional $\F_q$-subspace of $\mathbb{V}$ and $\langle \mathrm{Fix}(T) \rangle_{\F_{q^m}}=\mathbb{V}$.
\end{theorem}

%

\medskip

\emph{Proof of Theorem \ref{main}}\\
\emph{1.} Let $G(x)=\sum_{i=0}^{t-1} b_ix^{q^{ni}}$ and $H=G\circ \sigma$. Note that $H$ is an $\F_{q^n}$-semilinear transformation of $\F_{q^{nt}}$ with companion automorphism $\sigma$.
Since $\sigma|_{\F_{q^n}}\colon\F_{q^n}\rightarrow\F_{q^n}$ has order $n$, it follows that $H^n$ is an $\F_{q^n}$-linear transformation of $\F_{q^{nt}}$.
Also, $E_1(H)=\{y \in \F_{q^{nt}} \colon H(y)=y\}$ coincides with the kernel of $f$, the subspace $E_1(H^n)=\{y \in \F_{q^{nt}} \colon H^n(y)=y\}$ is an $\F_{q^n}$-subspace of $\F_{q^{nt}}$ and
\begin{equation}\label{A1(H)A1(Hn)}
E_1(H)\subseteq E_1(H^n).
\end{equation}
Since $H$ is an $\F_{q^n}$-semilinear transformation with companion automorphism $\sigma$ and for each $\lambda\in\F_{q^n}$ the condition $\lambda^\sigma=\lambda$ implies $\lambda \in \F_q$, by induction it is easy to see that if $y_1,\ldots,y_h\in E_1(H)$ are $\F_q$-independent then $y_1,\ldots,y_h$ are also $\F_{q^n}$-independent.
As a consequence we get the first point of the assertion, i.e. $\dim_{\F_q} \ker f \leq \dim_{\F_{q^n}} E_1(H^n)\leq t$.

\smallskip

\emph{2.} We have to prove that $\ker f \neq \{0\}$ if and only if there exists $y \in \F_{q^{nt}}^*$ such that $(G\circ\sigma)^n(y)=y$.
If $y \in \ker f$ with $y \neq 0$, then clearly $H^n(y)=y$, since $H(y)=y$.
Now, suppose that there exists $y\neq 0$ with $H^n(y)=y$.
Note that, we may write
\[ H^n-\mathrm{id}=(H-\mathrm{id})\circ(H^{n-1}+H^{n-2}+\ldots+\mathrm{id}). \]
Let $L=H^{n-1}+H^{n-2}+\ldots+\mathrm{id}$, which is an $\F_q$-linear transformation of $\F_{q^{nt}}$.
So,
\[0=(H^n-\mathrm{id})(y)=(H-\mathrm{id})(L(y)),\]
and hence $L(y)\in\ker(H-\mathrm{id})=\ker f$.
Furthermore, if $y \in \ker(H^n-\mathrm{id})$ then $\lambda y \in  \ker(H^n-\mathrm{id})$ for each $\lambda \in \F_{q^n}$, since $H^n-\mathrm{id}$ is $\F_{q^n}$-linear.
Hence, if $y \in \ker(H^n-\mathrm{id})$ then $L(\lambda y) \in \ker(H-\mathrm{id})=\ker f$ for each $\lambda \in \F_{q^n}$.
Since
\[ L(\lambda y)=H^{n-1}(\lambda y)+\dots+H(\lambda y)+\lambda y= \]
\[ =\lambda^{\sigma^{n-1}}H^{n-1}(y)+\ldots+\lambda^\sigma H(y)+\lambda y, \]
by Theorem \ref{Gow}, it follows that $L(\lambda y)$ cannot be zero for each $\lambda \in \F_{q^n}$.
So, for some $\lambda \in \F_{q^n}$, we have that $L(\lambda y)\in \ker f$ and $L(\lambda y)\neq 0$ and hence $\ker f \neq \{0\}$.

\smallskip

\emph{3.} Let $\dim_{\F_{q^n}} E_1(H^n)=h$ with $1\leq h\leq t$.
If $y\in E_1(H^n)$, then
\[H^n(H(y))=H(H^n(y))=H(y)\]
and so $H(E_1(H^n))\subseteq E_1(H^n)$.
Hence, we may consider
\[ H^*\colon y\in E_1(H^n) \mapsto H(y) \in E_1(H^n),\]
which is an $\F_{q^n}$-semilinear transformation of $E_1(H^n)=V(h,q^n)$.
If $H^*(y)=0$, then $H(y)=0$ and hence $H^n(y)=y=0$, since $y \in E_1(H^n)$.
It follows that $H^*$ is an $\F_{q^n}$-semilinear invertible transformation of $E_1(H^n)=V(h,q^n)$ with companion automorphism $\sigma$.
Since $\sigma|_{\F_{q^n}}$ has order $n$ and $H^{*n}(y)=H^n(y)=y$ for each $y \in E_1(H^n)$, it follows that $H^*$ has order $n$.
So, by Theorem \ref{fixsemi}, we have that $\dim_{\F_q} E_1(H)=\dim_{\F_{q^n}} E_1(H^n)$ and
\[ \langle E_1(H)\rangle_{\F_{q^n}}=E_1(H^n), \]
and hence
\[ \dim_{\F_q} \ker f=\dim_{\F_{q}} E_1(H)=\dim_{\F_{q^n}} E_1(H^n)=\dim_{\F_{q^n}}(\ker((G\circ\sigma)^n-\mathrm{id})). \]
\qed

\section{Method of finding the roots}\label{sec:findingroots}

In this section we develop a method for finding roots of polynomials of Form \eqref{form}; indeed, we prove that in order to find the roots of a polynomial of Form \eqref{form} we just need to find the roots of a $q^n$-polynomial.

\begin{theorem}\label{findroots}
Let
\[ f(x)=-x+b_0x^{\sigma}+b_1x^{\sigma q^n}+b_2x^{\sigma q^{2n}}+\ldots+b_{t-1}x^{\sigma q^{n(t-1)}} \in \mathcal{L}_{nt,q}, \]
where $\sigma$ a generator of the Galois group $\mathrm{Gal}(\F_{q^n}\colon\F_q)$.
Let consider $G(x)=\sum_{i=0}^{t-1} a_ix^{q^{ni}}$, $H=G\circ\sigma$ and $L=H^{n-1}+H^{n-2}+\ldots+H+\mathrm{id}$.
Then
\[ \ker f=L(E_1(H^n)). \]
\end{theorem}
\begin{proof}
As already seen in the proof of Theorem \ref{main}, we have that
\[ H^n-\mathrm{id}=(H-\mathrm{id})\circ L, \]
where $L=H^{n-1}+\ldots+H+\mathrm{id}$ and if $y \in E_1(H^n)$ then $L(y)\in E_1(H)$.
Now, consider
\[ L^*\colon y \in E_1(H^n) \mapsto L(y) \in E_1(H), \]
which is an $\F_q$-linear map.
If $\lambda \in \F_{q^n}$ and $y \in E_1(H)$ then $\lambda y\in E_1(H^n)$ and
\[ L^*(\lambda y)=(\lambda^{\sigma^{n-1}}+\ldots+\lambda^\sigma+\lambda)y=\mathrm{Tr}_{q^n/q}(\lambda)y, \]
which implies that $E_1(H) \subseteq \mathrm{Im} L^*$ and so $L^*(E_1(H^n))=E_1(H)=\ker f$.
\end{proof}


Let see some working examples, in which we show how to use our result.

\begin{example}\label{ex:trin}
Let us consider $q=p^h$, $n=3$, $t\geq 2$ and $x^\sigma=x^{q}$, hence
\[ f(x)=-x-x^q+x^{q^4}\in\tilde{\mathcal{L}}_{3t,q}, \]
is a polynomial of Form \eqref{form}.
By Theorem \ref{main}, it follows that
\[\dim_{\F_q} \ker f=\dim_{\F_{q^3}} \ker (H^3-\mathrm{id}),\]
where $H(x)=-x^q+x^{q^4}$, and
\[ H^3(x)-x=x^{q^{12}}-3x^{q^9}+3x^{q^6}-x^{q^3}-x \in \mathcal{L}_{{3t},q}. \]
Also,
\[ L(x)=H^2(x)+H(x)+x=x^{q^8}-2x^{q^5}+x^{q^4}+x^{q^2}-x^q+x \in \mathcal{L}_{{3t},q}. \]
Therefore, by Theorem \ref{findroots} we have the following
\[ \ker f=\{x_0^{q^8}-2x_0^{q^5}+x_0^{q^4}+x_0^{q^2}-x_0^q+x_0 \colon x_0 \in \F_{q^{3t}}\,\,\text{and}\,\, H^3(x_0)=x_0\}. \]
The trivial upper bound for the dimension of the kernel of $f$ is $\dim_{\F_q} \ker f \leq 4$ and this bound can be reached. Indeed, choosing $t=5$ and $p=2$ then we have that $H^3(x)-x=x^{q^{12}}+x^{q^9}+x^{q^6}+x^{q^3}+x=\tr_{q^{15}/q^3}(x)$ modulo $x^{q^{15}}-x$ and so, in such a case,
\[ \dim_{\F_q} \ker f=4. \]
Also,
\[ \ker f=\{x_0^{q^8}+x_0^{q^4}+x_0^{q^2}+x_0^q+x_0 \colon x_0 \in \F_{q^{15}}\,\,\text{and}\,\, \tr_{q^{15}/q^3}(x_0)=0\}. \]
Suppose that $t=4$, then $H^3(x)-x$ modulo $x^{q^{12}}-x$ is
\[ H^3(x)-x=(-3x^{q^6}+3x^{q^3}-x)^{q^3} \]
and so
\[ \dim_{\F_{q^3}} \ker (H^3(x)-x)=\dim_{\F_{q^3}} \ker(-3x^{q^6}+3x^{q^3}-x)\leq 2. \]
Since
\[D(-3x^{q^6}+3x^{q^3}-x)=\begin{pmatrix} -1 & 3 & -3 & 0 \\ 0 & -1 & 3 & -3\\ -3 & 0 & -1 & 3 \\ 3 & -3 & 0 & -1 \end{pmatrix}, \]
then $\det D(-3x^{q^6}+3x^{q^3}-x)=7\cdot 13$ and $\det D_1(-3x^{q^6}+3x^{q^3}-x)=3^2$ and by using Theorems \ref{bence} and \ref{findroots}, we get that
\[ \dim_{\F_q} \ker f=\dim_{\F_{q^3}} \ker(-3x^{q^6}+3x^{q^3}-x)=\left\{ \begin{array}{llll} 1 & \text{if}\,\,p=7,13; \\ 0 & \text{if}\,\,p\neq7,13. \end{array} \right. \]
So, if $p\neq7,13$ the polynomial $f$ is a permutation polynomial and if either $p =7$ or $p=13$, then $\dim_{\F_q} \ker f=1$.

When $t=3$, then $H^3(x)-x=3x^{q^6}-4x$ seen modulo $x^{q^9}-x$ and $3x^{q^6}-4x=0$ for $x\neq0$ if and only if $x^{q^3-1}=\frac{3}4$, which admits $q^3-1$ roots if $\N_{q^9/q^3}(3/4)=1$ and zero solutions otherwise.
Since $\N_{q^9/q^3}(3/4)=1$ if and only if $p=37$, by Theorem \ref{main} we have that
\[ \dim_{\F_q}\ker f=\left\{ \begin{array}{lll} 1 & \text{if}\,\, p=37\\ 0 & \text{if}\,\, p \neq 37 \end{array}\right.. \]
Let $p=37$ and let $x_0 \in \F_{q^9}^*$ such that $x_0^{q^3-1}=\frac{3}4$, then
\[ L(x_0)=H^2(x_0)+H(x_0)+x_0=\frac{1}{16}x_0^{q^2}-\frac{1}4x_0^q+x_0, \]
and so by Theorem \ref{findroots}
\[ \ker f=\left\{\frac{1}{16}x_0^{q^2}-\frac{1}4x_0^q+x_0 \colon x_0\in \F_{q^9} \,\,\text{and}\,\, x_0^{q^3}=\frac{3}4x_0 \right\}. \]

When $t=2$, we have that $H^3(x)-x=-4x^{q^3}+3x$ modulo $x^{q^6}-x$ and as before
\[ \dim_{\F_q}\ker f=\left\{ \begin{array}{lll} 1 & \text{if}\,\, p=37\\ 0 & \text{if}\,\, p \neq 37 \end{array}\right., \]
and also
\[ \ker f=\left\{\frac{1}{16}x_0^{q^2}-\frac{1}4x_0^q+x_0 \colon x_0\in \F_{q^6} \,\,\text{and}\,\, x_0^{q^3}=\frac{3}4x_0 \right\}. \]
\end{example}

In the next section in Theorem \ref{thm:maintri} we will deal with trinomials in a more general fashion.

\section{Trinomials}\label{sec:trin}

Very recently, in \cite[Theorem 1.1]{McGuireMueller}, McGuire and Mueller provide a full characterization of linearized trinomials of the form $f(x)=ax+bx^q+x^{q^d} \in \tilde{\mathcal{L}}_{m,q}$ that split completely over $\mathbb{F}_{q^m}$ when $m\leq d^2-d+1$.
In particular, when $m\leq d(d-1)$ and $d$ does not divide $m$, they prove that $\dim_{\F_q} \ker f<d$.
Their results have been already used in \cite{TrombZullo} for decodability issue of rank metric codes and this motivates to study similar results for trinomials of Form \eqref{form}.
More precisely, we deal with trinomials of Form \eqref{form}, i.e.

\begin{equation}\label{eq:trin}
f(x)=-x+ax^\sigma+bx^{\sigma q^{\ell n}} \in \tilde{\mathcal{L}}_{nt,q},
\end{equation}
where $\sigma$ is a generator of $\mathrm{Gal}(\F_{q^n}\colon \F_q)$, $a,b \in \F_{q^{nt}}$ and $1\leq \ell \leq t-1$.

We assume that $a$ and $b$ are nonzero, in order to avoid trivial cases.

\begin{theorem}\label{thm:maintri}
Let
\[ f(x)=-x+ax^\sigma+bx^{\sigma q^{\ell n}} \in \tilde{\mathcal{L}}_{nt,q}, \]
where $\sigma$ is a generator of $\mathrm{Gal}(\F_{q^n}\colon\F_q)$, $x^\sigma=x^{q^s}$ with $\gcd(s,n)=1$ and $a,b\neq 0$.
Then
\[\dim_{\F_q} \ker f\leq t.\]
Moreover, if $t\leq n\ell+s$ and one of the following conditions hold:
\begin{enumerate}
  \item $s+h\ell \not\equiv0 \pmod{t}$, for each $0\leq h\leq n$;
  \item $j\ell \not\equiv 0 \pmod{t}$, for each $1\leq j\leq n$, and $s\not\equiv 0\pmod{t}$;
  \item $\ell n\not\equiv i\ell \pmod{t}$, for each $0\leq i\leq n-1$, and $s+\ell n\not\equiv 0\pmod{t}$;
\end{enumerate}
then
\begin{equation}\label{eq:min}
\dim_{\F_q}\ker f \leq \min\{t-1,(n-1)\ell+s\}.
\end{equation}
\end{theorem}
\begin{proof}
By Theorem \ref{main}, we know that $\dim_{\F_q} \ker f\leq t$ and $\dim_{\F_q} \ker f=\dim_{\F_{q^n}} \ker (H^n-\mathrm{id})$,
where $H(x)=ax^\sigma+bx^{\sigma q^{\ell n}}$.
Hence,
\[ H^n(x)-x=-x+\alpha_0 x^{\sigma^n}+\alpha_1 x^{\sigma^n q^{\ell n}}+\ldots+\alpha_{n-1}x^{\sigma^n q^{(n-1)n\ell}}+\alpha_nx^{\sigma^n q^{n^2\ell}}\in \mathcal{L}_{t,q^n},\]
i.e.
\[ H^n(x)-x=-x+\alpha_0 x^{q^{ns}}+\alpha_1 x^{ q^{n(s+\ell)}}+\ldots+\alpha_nx^{q^{n(s+n\ell)}},\]
where $\alpha_0= a^{1+\sigma+\ldots+\sigma^{n-1}}$ and $\alpha_n=b^{1+\sigma q^{n\ell}+\ldots+\sigma^{n-1}q^{(n-1)\ell n}}$.
In order to improve the bound of Theorem \ref{main}, we show that the polynomial $H^n(x)-x \pmod{x^{q^{nt}}-x}$ is not the zero polynomial when one among Assumptions 1., 2. and 3. hold.
Indeed, if at least one of them holds, then reducing $H^n(x)-x$ modulo $x^{q^{nt}}-x$ at least one of the monomials $x$, $x^{q^s}$  or $x^{q^{n(s+n\ell)}}$ has nonzero coefficient.
Hence,
\[ \dim_{\F_q}\ker f \leq t-1. \]
Note that the $q^n$-degree of $H^n(x)-x\in\mathcal{L}_{t,q^n}$ is $s+\ell n$.
Also, if $t\leq s+\ell n$, then the $q^n$-degree of $H^n(x)-x \pmod{x^{q^{nt}}-x}$ is less than $(n-1)\ell+s$.
Indeed,
\[ \alpha_n x^{q^{n(s+n\ell)}}=\alpha_nx^{q^{ng}} \pmod{x^{q^{nt}}-x}, \]
with $s+n\ell=tk+g$ for some positive integer $k\geq 1$.
If $g\geq s+(n-1)\ell=tk+g-\ell$, then
\[ \ell \geq tk \geq t, \]
which is a contradiction as $\ell<t$.
Therefore, the $q^n$-degree of $H^n(x)-x \pmod{x^{q^{nt}}-x}$ is less than or equal to $(n-1)\ell+s$ and so \eqref{eq:min} holds.
\end{proof}

Let consider the following trinomial
\[ f(x)=-x+ax^{q^s}+b x^{q^h} \in \tilde{\mathcal{L}}_{m,q}, \]
with $a,b \neq 0$.
We may use our results to get relevant information when $\gcd(s,m)=1$, $m=nt$ and $n \equiv h \pmod{s}$.
In particular, under the above assumptions, we have that $h=n+\ell s$, for some positive integer $\ell$ and if $1 \leq \ell \leq t-1$ and at least one of the Assumptions 1., 2. or 3. hold, then $\dim_{\F_q} \ker f \leq t-1$.
When $s=1$, we may compare our results with the above mentioned results of  McGuire and Mueller.
The following example shows that, as long as the extension degree involved is not too large, the previous theorem can improve in some particular cases the results of McGuire and Mueller cited above.

\begin{example}
Consider
\[-x+ax^q+bx^{q^7} \in \tilde{\mathcal{L}}_{3t,q},\]
with $a,b \neq 0$.
Results of \cite{McGuireMueller} imply that if $t\leq 14$ and $t\neq 7,14$, then $\dim_{\F_q} \ker f<7$.
Applying Theorem \ref{thm:maintri} to these trinomials with $s=1$, $n=3$ and $\ell=2$, we get that if $4\leq t\leq 7$, then
\[\dim_{\F_q} \ker f \leq \min\{t-1,5\}.\]
\end{example}

\begin{example}
Consider
\[-x+ax^{q^2}+bx^{q^{11}} \in \tilde{\mathcal{L}}_{3t,q},\]
with $a,b \neq 0$.
Results of \cite{McGuireMueller} cannot be applied for this polynomial, whereas Theorem \ref{thm:maintri} with $s=2$, $n=3$ and $\ell=3$ implies that for each $a,b \in \F_{q^n}^*$
\[\dim_{\F_q} \ker f \leq \min\{t-1,8\},\]
when $3\leq t\leq 11$.
\end{example}

\section{Proof of Theorem \ref{main1}}\label{sec:matrix}

Here, we present results of the form \cite{Cs2019,teoremone,McGuireSheekey}, i.e. we characterize the number of roots of a $q$-polynomial of Form \eqref{form} by giving relations on its coefficients and involving a much smaller matrix.

Let denote by $\tau_{q^i}$ the automorphism of $\F_{q^{nt}}$ defined as $\tau_{q^i} (x)=x^{q^i}$.\\
\noindent The following remark will be useful in the sequel.

\begin{remark}\label{DicksonAlgebra}
In \cite{wl}, the authors prove the existence of an isomorphism between the $\F_q$-algebra $\tilde{\mathcal{L}}_{m,q}$ and the $\F_q$-algebra of Dickson matrices of order $m$ over $\F_{q^m}$. Here, we point out some properties proved in \cite{wl}:
\begin{itemize}
  \item $D(f+g)=D(f)+D(g)$, for $f,g \in \tilde{\mathcal{L}}_{m,q}$;
  \item $D(f\circ g)=D(f)\cdot D(g)$, for $f,g \in \tilde{\mathcal{L}}_{m,q}$;
  \item if $f(x)=\sum_i a_i x^{q^i}$ and $g(x)=\tau\circ f\circ \tau^{-1}(x)=\sum_i a_i^\tau x^{q^i}$, with $\tau \in \mathrm{Aut}(\F_{q^m})$, then $D(g)=D(f)^\tau$.
\end{itemize}
\end{remark}

\emph{Proof of Theorem \ref{main1}}\\
As already observed, $f=G\circ \sigma-\mathrm{id}$, where $G$ is the $\F_{q^n}$-linear map of $\F_{q^{nt}}$ defined by the rule $G(x)=\sum_{i=0}^{t-1} b_ix^{q^{ni}}$. Denote by
\[G^{\sigma^{-1}}:=\sigma^{-1}\circ G\circ \sigma,\]
and note that $G^{\sigma^{-1}}(x)=\sum_{i=0}^{t-1} b_i^{\sigma^{-1}}x^{q^{ni}}$.
Then $G\circ\sigma=\sigma\circ G^{\sigma^{-1}}$ and for each positive integer $i$ we have that $G\circ \sigma^i=\sigma^i\circ G^{\sigma^{-i}}$.
Now, we show that
\[H^\ell=(G\circ\sigma)^\ell=\sigma^{\ell-1}\circ G^{\sigma^{-(\ell-1)}}\circ\ldots\circ G\circ \sigma,\]
for each positive integer $\ell$.
Clearly,
\[ H^2=(G\circ\sigma)^2=\sigma \circ G^{\sigma^{-1}} \circ G \circ \sigma. \]
Suppose that for $\ell\geq2$, $H^{\ell-1}=\sigma^{\ell-2}\circ G^{\sigma^{-(\ell-2)}}\circ\ldots\circ G\circ \sigma$, then
\[ H^\ell=(G\circ \sigma)\circ(G\circ\sigma)^{\ell-1}=(G\circ \sigma)\circ (\sigma^{\ell-2}\circ G^{\sigma^{-(\ell-2)}}\circ\ldots\circ G\circ \sigma)= \]
\[ =\sigma^{\ell-1}\circ G^{\sigma^{-(\ell-1)}}\circ G^{\sigma^{-(\ell-2)}}\circ\ldots\circ G\circ \sigma. \]
Hence, $H^n=(G\circ\sigma)^n=\sigma^{n-1}\circ G^{\sigma^{-(n-1)}}\circ\ldots\circ G\circ \sigma$.
Also,
\[ \sigma\circ(H^n-\mathrm{id})\circ\sigma^{-1}=\sigma^n \overline{G}-\mathrm{id}, \]
where $\overline{G}=G^{\sigma^{-(n-1)}}\circ\ldots\circ G$.
Clearly,
\[ \dim_{\F_{q^n}} \ker (H^n-\mathrm{id})=\dim_{\F_{q^n}} \ker(\sigma^n \overline{G}-\mathrm{id})=\dim_{\F_{q^n}} \ker(\overline{G}-\sigma^{-n}). \]
Since $\tau=\sigma^{-1}$, then $\overline{G}=G^{\tau^{n-1}}\circ\ldots\circ G^\tau\circ G$ and $\overline{G}-\sigma^{-n}=\overline{G}-\tau^n$.
Note that $\overline{G}-\tau^n$ is an $\F_{q^n}$-linear transformation and by \cite{wl}, we have that
\[ \dim_{\F_{q^n}} \ker(\overline{G}-\tau^n)=t-\mathrm{rk}(D(\overline{G}-\tau^n)).\]
Also, since $D(\tau_{q^n})=J$ and $\tau^n=\tau_{q^{sn}}=(\tau_{q^n})^s$, by Remark \ref{DicksonAlgebra} it follows that
\[ D(\overline{G}-\tau^n)=D(G)^{\tau^{n-1}}\cdot\ldots\cdot D(G)-J^s, \]
and since $D(G)$ coincides with $D$ we have the assertion.
\qed

As a consequence of the previous result we can characterize permutation (i.e. invertible) linearized polynomials of Form \eqref{form} and we can characterize and give sufficient conditions on the case of maximum dimension of the kernel w.r.t. bound 1. of Theorem \ref{main}, similarly to \cite[Theorem 1.2]{teoremone} and \cite[Theorem 10]{GQ2009x}.

\begin{corollary}\label{corollary:main}
Let
\[ f(x)=-x+b_0x^{\sigma}+b_1x^{\sigma q^n}+b_2x^{\sigma q^{2n}}+\ldots+b_{t-1}x^{\sigma q^{n(t-1)}} \in \mathcal{L}_{nt,q}, \]
where $\sigma \in \mathrm{Aut}(\F_{q^{nt}})$ such that $\sigma|_{\F_{q^n}}\colon\F_{q^n}\rightarrow\F_{q^n}$ has order $n$.
Let $D$, $J$ and $s$ as in Theorem \ref{main1}. Then
\begin{itemize}
  \item $f(x)$ is a permutation polynomial if and only if
  \[\det (D^{\tau^{n-1}} \cdot D^{\tau^{n-2}} \cdot \ldots \cdot D-J^s) \neq 0;\]
  \item $\dim_{\F_q}\ker f=t$ if and only if
  \begin{equation}\label{matrixcondition2}
  D^{\tau^{n-1}} \cdot D^{\tau^{n-2}} \cdot \ldots \cdot D=J^s.
  \end{equation}
\end{itemize}
In particular, if $\dim_{\F_q}\ker f=t$, then $\N_{q^{tn}/q^n}(\det(D))=(-1)^{s(t-1)}$.
\end{corollary}

If the $q^n$-polynomial $G$ has non-trivial kernel we can improve the bound on the dimension of the kernel of $f$.

\begin{corollary}\label{cor:Gnoinv}
Let
\[ f(x)=-x+b_0x^{\sigma}+b_1x^{\sigma q^n}+b_2x^{\sigma q^{2n}}+\ldots+b_{t-1}x^{\sigma q^{n(t-1)}} \in \mathcal{L}_{nt,q}, \]
where $\sigma$ is a generator of $\mathrm{Gal}(\F_{q^n}\colon\F_q)$.
Let $D$, $G$, $J$ and $s$ as in Theorem \ref{main1}. Then
\[ \dim_{\F_q} \ker f \leq t- \dim_{\F_{q^n}} \ker G. \]
\end{corollary}
\begin{proof}
Let $h=\dim_{\F_{q^n}} \ker G$, $M=D^{\tau^{n-1}} \cdot D^{\tau^{n-2}} \cdot \ldots \cdot D$ and let $\overline{G}$ and $\overline{H}$ the $q^n$-polynomials such that $D(\overline{G})=M$ and $D(\overline{H})=M-J^s$, i.e. $\overline{H}=\overline{G}-\tau^n$.
By Theorem \ref{main} $\dim_{\F_q} \ker f=\dim_{\F_{q^n}} \ker \overline{H}$, since $\dim_{\F_{q^n}} \ker G=h$, we have that $\mathrm{rk}(D)=t-h$ and $\mathrm{rk}(M)\le t-h$.
Also, since $\tau^n$ is invertible, then $\ker \overline{H} \cap \ker \overline{G}=\{0\}$ and hence
\[ \dim_{\F_{q^n}} \ker \overline{H} + \dim_{\F_{q^n}} \ker \overline{G} \leq t, \]
i.e.
\[ \dim_{\F_q} \ker f=\dim_{\F_{q^n}} \ker \overline{H} \leq t-\dim_{\F_{q^n}} \ker \overline{G}=\mathrm{rk}(M)\leq t-h. \]
\end{proof}

\begin{remark}\label{adjoint}
The \emph{adjoint} of a $q$-polynomial $f(x)=\sum_{i=0}^{n-1}a_i x^{q^i}$, with respect to the bilinear form $\langle x,y\rangle:=\mathrm{Tr}_{q^n/q}(xy)$, is given by
\[\hat{f}(x):=\sum_{i=0}^{n-1}a_{i}^{q^{n-i}} x^{q^{n-i}}.\]
In particular, if $f(x)$ is a $q$-polynomial of Form \eqref{form}, then
\[ f(x)=ax+b_0x^{q^s}+b_1x^{q^{n+s}}+b_2x^{q^{2n+s}}+\ldots+b_{t-1}x^{ q^{n(t-1)+s}}\in \tilde{\mathcal{L}}_{nt,q}, \]
with $\gcd(s,n)=1$ and its adjoint is
\[ \hat{f}(x)=ax+b_0^{q^{nt-s}}x^{q^{nt-s}}+b_1^{q^{n(t-1)-s}}x^{q^{n(t-1)-s}}+b_2^{q^{n(t-2)-s}} x^{q^{n(t-2)-s}}+\ldots+b_{t-1}^{ q^{n-s}}x^{ q^{n-s}}, \]
i.e. $\hat{f}(x)$ is of Form \eqref{form} with $\sigma=\tau_{q^{n-s}}$.
Therefore, the family of $q$-polynomials we are studying is closed by the adjoint operation.
Furthermore, we underline that by \cite[Lemma 2.6]{BGMP2015}, see also \cite[pages 407--408]{CsMP}, the kernels of $f$ and $\hat{f}$ have the same dimension and hence we may study this class up to the adjoint operation.
\end{remark}

\subsection{Recursive relations for the maximal case}

As in \cite{teoremone}, we show that Equality \eqref{matrixcondition2} holds if and only if $\mathbf{e}_0=(1,0,\ldots,0)$ is sent by $M=D^{\tau^{n-1}} \cdot \ldots \cdot D$ in a particular vector, which implies less conditions to manage in the maximal case.

\begin{lemma}\label{tcond}
Let $D$ be the matrix as in Theorem \ref{main1}.
Equality \eqref{matrixcondition2} holds if and only if
\[\mathbf{e}_0 D^{\tau^{n-1}} \cdot D^{\tau^{n-2}} \cdot \ldots \cdot D =\mathbf{e}_{r},\]
where $r \equiv s \pmod{t}$ and $\mathbf{e}_i$ is the vector of $\F_{q^{nt}}^t$ whose $i$-th component is one and all the others are zero.
\end{lemma}
\begin{proof}
As seen in the proof of Theorem \ref{main1}, the matrices $M=D^{\tau^{n-1}} \cdot D^{\tau^{n-2}} \cdot \ldots \cdot D$ and $J^s$ are the Dickson matrices of two $q^n$-polynomials, hence they are autocirculant.
Therefore,  $D^{\tau^{n-1}} \cdot D^{\tau^{n-2}} \cdot \ldots \cdot D =J^s$ if and only if they coincide on the first row, i.e.
\[ \mathbf{e}_0 D^{\tau^{n-1}} \cdot D^{\tau^{n-2}} \cdot \ldots \cdot D =\mathbf{e}_0 J^s=\mathbf{e}_r. \]
\end{proof}

By Lemma \ref{tcond}, to describe recursively the relations on the coefficients of $f(x)$ characterizing the case in which the kernel of $f$ has dimension $t$, we need just to multiply $D^{\tau^{n-1}} \cdot D^{\tau^{n-2}} \cdot \ldots \cdot D$ by $\mathbf{e}_0$ or, equivalently,
\begin{equation}\label{eq:DDtau}
D^T \cdot (D^{\tau})^T \cdot \ldots \cdot (D^{\tau^{n-1}})^T\mathbf{e}_0^T =\mathbf{e}_r^T.
\end{equation}
Let $\phi$ be the $\F_{q^n}$-semilinear transformation having $D^T$ as associated matrix w.r.t. the canonical basis and $\tau$ as the companion automorphism. Then \eqref{eq:DDtau} holds if and only if $\phi^n(\mathbf{e}_0)=\mathbf{e}_r$.

We have $\phi(\mathbf{e}_0)=(b_0,\ldots,b_{t-1})$ where $f(x)=-x+b_0x^{\sigma}+b_1x^{\sigma q^n}+b_2x^{\sigma q^{2n}}+\ldots+b_{t-1}x^{\sigma q^{n(t-1)}} $, and for $i\geq 1$ let
\[ \phi^i(\mathbf{e}_0)=(P_{0,i},\ldots,P_{t-1,i}), \]
where $P_{j,i}$ is seen as a polynomial in $\F_{q^{tn}}$ in the variables $b_0,\ldots,b_{t-1}$ with $j \in \{0,\ldots,t-1\}$, then
\[ \phi^{i+1}(\mathbf{e}_0)=\left( \left(\begin{array}{cccccccccccc}
	b_0 & b_{t-1}^{q^n} & \cdots & b_1^{ q^{n(t-1)}} \\
	\vdots & \vdots & & \vdots &  \\
	b_{t-1} & b_{t-2}^{ q^n} & \cdots & b_0^{ q^{n(t-1)}}
	\end{array}\right) \left( \begin{array}{lllr} P_{0,i}^\tau \\ \vdots \\P_{t-1,i}^\tau \end{array} \right) \right)^T= \]
\[ =(b_0 P_{0,i}^\tau+b_{t-1}^{ q^n}P_{1,i}^\tau+\ldots+b_1^{ q^{n(t-1)}}P_{t-1,i}^\tau,\ldots,b_{t-1} P_{0,i}^\tau+b_{t-2}^{q^n}P_{1,i}^\tau+\ldots+b_0^{q^{n(t-1)}}P_{t-1,i}^\tau ). \]

Therefore, we can define recursively the polynomials $P_{j,i}$ as follows; for $i=1$
\[ P_{0,1}=b_0, \,\,\,\,\,\ldots\,\,\,\,\,\,\, P_{t-1,1}=b_{t-1}, \]
and for $i\geq 2$
\[ \begin{array}{lll} P_{0,i}= b_0 P_{0,i-1}^\tau+b_{t-1}^{q^n}P_{1,i-1}^\tau+\ldots+b_1^{ q^{n(t-1)}}P_{t-1,i-1}^\tau, \\ \vdots \\ P_{t-1,i}=b_{t-1} P_{0,i-1}^\tau+b_{t-2}^{ q^n}P_{1,i-1}^\tau+\ldots+b_0^{q^{n(t-1)}}P_{t-1,i-1}^\tau. \end{array} \]

As a consequence of Corollary \ref{corollary:main} and Lemma \ref{tcond}, we have the following result.

\begin{corollary}\label{cor:cond}
The dimension of the kernel of $f(x)$ is $t$ if and only if
\[ P_{j,n}=\left\{ \begin{array}{llrr} 1 & \text{if}\,\, j=r \\ 0 & \text{otherwise} \end{array} \right., \]
where $r \equiv s \pmod{t}$.
\end{corollary}

\section{Criteria for $t=2$}\label{sec:criteria}

In this section we will deal with polynomials of this form
\begin{equation}\label{formt=2}
f(x)=-x+b_0x^\sigma+b_1x^{\sigma q^n} \in \tilde{\mathcal{L}}_{2n,q},
\end{equation}
with $\sigma$ a generator of $\mathrm{Gal}(\F_{q^{n}}:\F_q)$ and $b_0, b_1\neq0$. Let $\tau=\sigma^{-1}$ and let $s$ be the minimum positive integer such that $\tau=\tau_{q^s}$.
We may assume w.l.o.g. that $s$ is odd. Indeed, if $s$ is even then we may consider
\[f(x)=-x+b_1x^{\sigma'}+b_0x^{\sigma' q^n},\]
with $\sigma'=\tau_{q^{n+s}}$ and $\gcd(s+n,2n)=1$.
By Theorem \ref{main} it follows that
\[ \dim_{\F_q} \ker\,f\leq 2, \]
and by Corollary \ref{corollary:main} we have that $\dim_{\F_q} \ker\,f= 2$ if and only if \eqref{matrixcondition2} holds.
Therefore, $\dim_{\F_q} \ker\,f= 2$ if and only if
\begin{equation}\label{matrixconditiont=2}
D^{\tau^{n-1}} \cdot D^{\tau^{n-2}} \cdot \ldots \cdot D=J^s=J,
\end{equation}
where $D=\left( \begin{array}{llrr} b_0 & b_1\\ b_1^{q^n} & b_0^{q^n}  \end{array} \right)$ and $J=\left( \begin{array}{llrr} 0 & 1 \\ 1 & 0 \end{array} \right)$.

Since $\det(D)=b_0^{q^n+1}-b_1^{q^n+1}$, by \eqref{matrixconditiont=2} we have the following result.

\begin{proposition}
If $f(x)=-x+b_0x^\sigma+b_1x^{\sigma q^n} \in \mathcal{L}_{2n,q}$ has kernel of dimension two then
\[ (b_0^{q^n+1}-b_1^{q^n+1})^{1+\tau+\ldots+\tau^{n-1}}=-1, \]
and hence $\N_{q^n/q}(b_0^{q^n+1}-b_1^{q^n+1})=-1$.
\end{proposition}

As a consequence of Corollaries \ref{corollary:main} and \ref{cor:cond}, we have the following result.

\begin{corollary}\label{cor:cond}
The dimension of the kernel of $f(x)$ is two if and only if
\[ \left\{ \begin{array}{llr} P_{0,n}=0 \\ P_{1,n}=1 \end{array} \right.. \]
\end{corollary}

\subsection{The $n=2$ case}

We are going to find more explicit relations on the coefficients of
\[ f(x)=-x+b_0x^\sigma+b_1x^{\sigma q^2} \in \mathcal{L}_{4,q}, \]
with $\sigma$ a generator of $\mathrm{Gal}(\F_{q^{2}}:\F_q)$, that completely characterize the dimension of the kernel of $f$.
The polynomial $f(x)$ is
either
\[ f_1(x)=-x+b_0x^q+b_1x^{q^3} \]
or
\[ f_2(x)=-x+b_0x^{q^3}+b_1x^q. \]

So, we may suppose that $f(x)=-x+b_0x^{q^3}+b_1x^q$ and hence $\tau=\tau_q$ and $s=1$.
By Corollary \ref{cor:cond}, we have that $f(x)$ has kernel of dimension two if and only if
\[ \left\{ \begin{array}{llr} P_{0,2}=0\\ P_{1,2}=1 \end{array} \right., \]
i.e.
\begin{equation}\label{condt=2n=2}
\left\{ \begin{array}{llr} b_0^{1+q}+b_1^{q^2+q}=0\\ b_1b_0^{q}+b_0^{q^2}b_1^{q}=1 \end{array} \right..
\end{equation}

From the previous equations we get that $b_0^{q+1}=-b_1^{q^2+q}$ and hence $\N_{q^4/q}(b_0)=\N_{q^4/q}(b_1)$.
In particular, $b_0$ and $b_1$ are nonzero.

\begin{proposition}\label{2dimt=2n=2}
The polynomial $f(x)$ has kernel of dimension two if and only if, denoting by $z=b_0/b_1$,
\begin{equation}\label{condt=2n=22}
\left\{ \begin{array}{llr} \N_{q^4/q}(z)=1\\ b_1^{q+1}=\frac{1}{z^q-z^{q^2+q+1}} \end{array} \right.
\end{equation}
is satisfied.
In particular, $\N_{q^4/q^2}(z)\neq1$.
\end{proposition}
\begin{proof}
We have to show that Systems \eqref{condt=2n=2} and \eqref{condt=2n=22} are equivalent.
Indeed, substituting $z$ in \eqref{condt=2n=2}, we have
\[ \left\{ \begin{array}{llr} b_1^{q+1}z^{q+1}+b_1^{q^2+q}=0\\ z^qb_1^{q+1}+z^{q^2}b_1^{q^2+q}=1 \end{array} \right., \]
and hence if and only if
\[ \left\{ \begin{array}{llr} b_1^{q^2-1}=-z^{q+1}\\ b_1^{q+1}(z^q-z^{q^2+q+1})=1 \end{array} \right.. \]
Therefore $\N_{q^4/q^2}(z)\neq 1$ and the previous system can be written as follows
\begin{equation}\label{normcond}
\left\{ \begin{array}{llr} (b_1^{q+1})^{q-1}=-z^{q+1}\\ b_1^{q+1}=\frac{1}{z^q-z^{q^2+q+1}} \end{array} \right..
\end{equation}
Substituting the second equation into the first equation, we get that the previous system is equivalent to \eqref{condt=2n=22}, since also the equations of \eqref{condt=2n=22} implies the first equation of \eqref{normcond}.
\end{proof}

It is possible to find many different choices for $b_0$ and $b_1$ in a way that $\dim_{\F_q}\ker\,f=2$, as shown in the next result.

\begin{proposition}\label{condt=2n=2exmax}
For each $z\in \F_{q^4}$ such that $\N_{q^4/q}(z)=1$ and $\N_{q^4/q^2}(z)\neq1$ there exist $q+1$ elements $b_1\in \F_{q^4}$ such that $\dim_{\F_q}\ker\,f=2$, where $f(x)=-x+b_0x^{q^3}+b_1x^q$ and $b_0/b_1=z$.
\end{proposition}
\begin{proof}
Suppose that $z\in \F_{q^4}$ with $\N_{q^4/q}(z)=1$ and $\N_{q^4/q^2}(z)\neq1$, the assertion is equivalent to find $q+1$ solutions in $b_1$ of the System \eqref{condt=2n=22}.
Such a values for $b_1$ exist if and only if
\[ \left( \frac{1}{z^q-z^{q^2+q+1}} \right)^{\frac{q^4-1}{q+1}}=1, \]
which happens if and only if
\begin{equation}\label{findb1}
z^{q(q-1)(q^2+1)}(1-z^{q^2+1})^{(q-1)(q^2+1)}=1.
\end{equation}
Let $y=z^{q^2+1}$ and note that $y \in \F_{q^2}$. Therefore, we are looking for $y \in \F_{q^2}$ such that $y^{q+1}=1$, $y\neq1$ and
\[ y^{1-q}(1-y)^{2q-2}=1, \]
i.e.
\[ y^{q+1}(y^q-y)=y^q-y, \]
which is satisfied since $y^{q+1}=1$.
Therefore, if $z\in \F_{q^4}$ with $\N_{q^4/q}(z)=1$ and $\N_{q^4/q^2}(z)\neq1$ Equation \eqref{findb1} is satisfied and so for each of such $z$ it is possible to find $q+1$ values for $b_1$ satisfying System \eqref{condt=2n=22}.
\end{proof}

As a consequence of the previous results we have the following classification theorem relating the dimension of the kernel of polynomials of the form $-x+b_0x^{q^3}+b_1x^q$.

\begin{theorem}\label{t=2n=2class}
Let
\[ f(x)=-x+b_0x^{q^3}+b_1x^q \in \mathcal{L}_{4,q}. \]
Then
\begin{enumerate}
  \item $\dim_{\F_q}\ker\,f \leq 2$;
  \item $\dim_{\F_q}\ker\,f =2$ if and only if
        \begin{equation}\label{systbinn=3}\left\{ \begin{array}{llr} \N_{q^4/q}(z)=1\\ b_1^{q+1}=\frac{1}{z^q-z^{q^2+q+1}} \end{array}, \right.\end{equation}
        where $z=b_0/b_1$;
  \item $f$ is invertible if and only if
        \[(b_0^{1+q}+b_1^{q+q^2})^{q^2+1}\neq(-1+b_0^{q}b_1+b_0^{q^2}b_1^{q})^{q^2+1};\]
  \item $\dim_{\F_q} \ker\,f=1$ if and only if
        \[(b_0^{1+q}+b_1^{q+q^2})^{q^2+1}=(-1+b_0^{q}b_1+b_0^{q^2}b_1^{q})^{q^2+1},\]
        and \eqref{systbinn=3} is not satisfied.
\end{enumerate}
\end{theorem}
\begin{proof}
It follows by Theorems \ref{main} and \ref{main1} and by Proposition \ref{2dimt=2n=2}.
\end{proof}

\subsection{The $n=3$ case}

Let consider any trinomial in $\F_{q^{6}}$ of the form
\[ f(x)=ax+bx^{q^i}+cx^{q^j} \in \mathcal{L}_{6,q}, \]
for some positive integers $i$ and $j$, with $i<j$, and $a,b,c \neq 0$.
It easy to see that each of such polynomials can be written, up to raising to a suitable $q$-th power, up to the adjoint operation (see Remark \ref{adjoint}) and up to multiply by an element of $\F_{q^6}^*$(\footnote{All of these operations do not change the dimension of the kernel.}), either as $\sigma$-polynomial
\[ f_1(x)=ax+bx^\sigma+cx^{\sigma^2},\]
with $\sigma \in \{\tau_q,\tau_{q^2}\}$,
or
\[ f_2(x)=ax+bx^{\sigma}+cx^{\sigma q^3}, \]
with $\sigma\in \{\tau_q,\tau_{q^5}\}$,
or
\[ f_3(x)=a'x+b'x^{\sigma}+c'x^{\sigma q^2}+d'x^{\sigma q^4}, \]
with $\sigma =\tau_q$ and one of $a',b',c'$ and $d'$ is zero.
For the former case, we may use the techniques developed in \cite{teoremone,McGuireSheekey} for establishing its number of roots directly from its coefficients by using a $2\times 2$ matrix. For such polynomials
\[ \dim_{\mathrm{Fix}(\sigma)} \ker f_1 \leq 2. \]
For the second and third cases, by Theorem \ref{main} we get, respectively,
\[ \dim_{\F_q}\ker f_2 \leq 2 \]
and
\[ \dim_{\F_q}\ker f_3 \leq 3 \]
which is not a consequence of Theorem \ref{Gow}.

We are going to investigate the second case and, up to the operations already discussed, we may choose $\sigma=\tau_{q^5}$ and
\[ f(x)=-x+b_0x^{q^5}+b_1x^{q^2}. \]
In this case $\tau=\tau_q$ and $s=1$.
By Corollary \ref{cor:cond}, we have that $f(x)$ has kernel of dimension two if and only if
\[ \left\{ \begin{array}{llr} P_{0,3}=0\\ P_{1,3}=1 \end{array} \right., \]
i.e.
\begin{equation}\label{condt=2n=3}
\left\{ \begin{array}{llr} b_0(b_0^{q+q^2}+b_1^{q^4+q^2})+b_1^{q^3}(b_1^qb_0^{q^2}+b_0^{q^4}b_1^{q^2})=0 \\
b_1(b_0^{q+q^2}+b_1^{q^4+q^2})+b_0^{q^3}(b_1^qb_0^{q^2}+b_0^{q^4}b_1^{q^2})=1 \end{array} \right..
\end{equation}

We are able to manage these relations, getting the following result.

\begin{theorem}\label{thm:binq6}
Let
\[f(x)=-x+b_0x^{q^5}+b_1x^{q^2} \in \mathcal{L}_{6,q},\]
with $b_0,b_1\neq0$, let $\alpha=b_1/b_0$ and $\displaystyle A=\frac{-\alpha^{q^3+1}}{1-\alpha^{q^3+1}}$.
The following holds.
\begin{enumerate}
    \item If $\dim_{\F_q} \ker f=2$, then $\alpha^{q^3+1}\neq1$ and the equation
    \begin{equation}\label{eq:Y}
    Y^2-(\mathrm{Tr}_{q^3/q}(A)-1)Y+\mathrm{N}_{q^3/q}(A)=0,
    \end{equation}
    admits either one root over $\F_q$ or two roots in $\F_{q^2}\setminus\F_q$.
    \item If $\alpha^{q^3+1} \in \F_{q^{3}}\setminus\{0,1\}$ and $A$ is such that Equation \eqref{eq:Y} admits either one root over $\F_q$ or two roots in $\F_{q^2}\setminus\F_q$, then there exists $b_0 \in \F_{q^6}^*$ such that
    \[ \dim_{\F_q} \ker (-x+b_0x^{q^5}+b_1 x^{q^2})=2, \]
    where $b_1=\alpha b_0$.
\end{enumerate}
\end{theorem}
\begin{proof}
{\bf 1.} Suppose that $\dim_{\F_q} \ker f=2$, then $b_0$ and $b_1$ satisfy \eqref{condt=2n=3} and substituting $\alpha$ in \eqref{condt=2n=3} we get
\begin{equation}\label{condt=2n=3alpha}
\left\{ \begin{array}{llr}
b_0(b_0^{q+q^2}+b_0^{q^4+q^2}\alpha^{q^4+q^2})+b_0^{q^3}\alpha^{q^3}(\alpha^q b_0^{q^2+q}+b_0^{q^4+q^2}\alpha^{q^2})=0 \\
b_0\alpha(b_0^{q+q^2}+b_0^{q^4+q^2}\alpha^{q^4+q^2})+b_0^{q^3}(\alpha^q b_0^{q^2+q}+b_0^{q^4+q^2}\alpha^{q^2})=1 \end{array} \right..
\end{equation}
By multiplying the first equality of \eqref{condt=2n=3alpha} by $\alpha$ and by subtracting the second equality,
and by multiplying the second equality of \eqref{condt=2n=3alpha} by $\alpha^{q^3}$ and by subtracting the first equality, we get that \eqref{condt=2n=3alpha} is equivalent to
\[ \left\{ \begin{array}{ll} (\alpha^{q^3+1}-1)(b_0^{q^3+q^2+q}\alpha^q+b_0^{q^4+q^3+q^2}\alpha^{q^2})=-1\\
(\alpha^{q^3+1}-1)(b_0^{1+q+q^2}+b_0^{1+q^2+q^4}\alpha^{q^2+q^4})=\alpha^{q^3}\end{array} \right..\]
Hence $\alpha^{q^3+1}\neq 1$ and we may write
\[  \left\{ \begin{array}{ll} b_0^{q^3+q^2+q}\alpha^q+b_0^{q^4+q^3+q^2}\alpha^{q^2}=\frac{-1}{\alpha^{q^3+1}-1}\\
b_0^{1+q+q^2}+b_0^{1+q^2+q^4}\alpha^{q^2+q^4}=\frac{\alpha^{q^3}}{\alpha^{q^3+1}-1}\end{array} \right.. \]
Let $z=b_0^{1+q+q^2}$ and $x=b_0^{1+q^2+q^4}$ and note that $x\in \F_{q^2}$.
With this notation and by multiplying the second equation by $\alpha$, the previous system becomes
\begin{equation}\label{eq:xz}
\left\{ \begin{array}{lll} (z\alpha)^q+(z\alpha)^{q^2}=1-A\\
z\alpha +x\alpha^{1+q^2+q^4}=A\\ z=b_0^{1+q+q^2}, \quad x=b_0^{1+q^2+q^4}\end{array} \right..
\end{equation}
Now, let $T=z\alpha$ and let $Y=x\alpha^{1+q^2+q^4}$, then \eqref{eq:xz} implies
\begin{equation}\label{eq:TA}
 \left\{ \begin{array}{lll} T^q+T^{q^2}=1-A\\
T +Y=A\\
x=z^{q^2-q+1}\end{array} \right..
\end{equation}
By combining the first and the second equation, since $Y \in \F_{q^2}$, we get
\begin{equation}\label{eq:17}
\left\{ \begin{array}{lll} T=A-Y\\
Y+Y^q=\tr_{q^3/q}(A)-1\\
\left(\frac{T}\alpha\right)^{q^2-q+1}=\frac{Y}{\alpha^{1+q^2+q^4}}\end{array} \right..
\end{equation}
By combining the first and the third, we get
\[ \frac{(A-Y)^{q^2-q+1}}{\alpha^{q^2-q+1}}=\frac{Y}{\alpha^{1+q^2+q^4}}, \]
since $\alpha^{q^3+1}=\frac{A}{A-1}$ and using the second equality, we get
\[Y^2-(\mathrm{Tr}_{q^3/q}(A)-1)Y+\mathrm{N}_{q^3/q}(A)=0.\]
If $Y \in \F_q$, then $2Y=\tr_{q^3/q}(A)-1$ and this implies that the above equation has one root, precisely if $q$ is odd $Y=\frac{\tr(A)-1}2$ and if $q$ is even $Y=\sqrt{A^{1+q+q^2}}$.
Therefore, \eqref{eq:Y} admits either one root over $\F_q$ or two roots in $\F_{q^2}\setminus\F_q$.

{\bf 2.} Now, assume that $\alpha$ is an element of $\F_{q^6}$ such that $\alpha^{q^3+1} \in \F_{q^{3}}\setminus\{0,1\}$ and $A=\frac{-\alpha^{q^3+1}}{1-\alpha^{q^3+1}}$ is such that Equation \eqref{eq:Y} admits either one root over $\F_q$ or two roots in $\F_{q^2}\setminus\F_q$. Let $\overline{Y}$ be a root of \eqref{eq:Y} and let $\overline{T}=A-\overline{Y}$. Then $(\overline{Y},\overline{T})$ is a solution of \eqref{eq:17} and choosing $\overline{x}=\frac{\overline{Y}}{\alpha^{1+q^2+q^4}}$ and $\overline{z}=\frac{A-\overline{Y}}{\alpha}$, we get that \eqref{eq:xz} is satisfied if we can find $b_0$ such that \[\overline{z}=b_0^{1+q+q^2} \quad \text{and} \quad \overline{x}=b_0^{1+q^2+q^4},\]
i.e. if we can find $b_0\in \F_{q^6}^*$ such that
\[ b_0^{1+q+q^2}=\frac{A-\overline{Y}}\alpha \quad \text{and} \quad b_0^{1+q^2+q^4}=\frac{\overline{Y}}{\alpha^{1+q^2+q^4}}. \]
Our aim is to prove the existence of a such $b_0 \in \F_{q^6}^*$.
First, we observe that
\begin{equation}\label{eq:x=z}
\overline{x}=\overline{z}^{q^2-q+1}.
\end{equation}
Indeed, it is equivalent to
\[ \frac{\overline{Y}}{\alpha^{1+q^2+q^4}}=\left(\frac{A-\overline{Y}}{\alpha}\right)^{q^2-q+1}, \]
i.e.
\[ \overline{Y}A^q-\overline{Y}^{q+1}=\frac{A^q}{A^q-1}A^{q^2+1}-\frac{A^q}{A^q-1}[\overline{Y}^2-(A+A^{q^2})\overline{Y}], \]
which results to be verified because of \eqref{eq:Y}.

Since $\overline{x}\in \F_{q^2}$ and since
\[ \overline{z}^{\frac{q^6-1}{1+q+q^2}}=\overline{z}^{(q^3+1)\frac{q^3-1}{1+q+q^2}}=(\overline{x}^{q+1})^{q-1}=1, \]
there exist $a,b \in \F_{q^6}^*$ such that
\[ \overline{x}=a^{1+q^2+q^4} \quad \text{and} \quad \overline{z}=b^{1+q+q^2}. \]
By \eqref{eq:x=z}, it follows that $b=a\eta$, with $\N_{q^6/q^2}(\eta)=1$, i.e. $\overline{x}=b^{1+q^2+q^4}$.
Therefore, System \eqref{eq:xz} is satisfied for $b_0=b$ and hence, the polynomial $f(x)=-x+b_0x^{q^5}+b_1 x^{q^2}$, where $b_0=b$ and $b_1=\alpha b$, has kernel with dimension $2$.
\end{proof}

\begin{remark}
With the above notation, $b_0$ mentioned in 2. Theorem \ref{thm:binq6} is a root of
\[ x^{q^2+q+1}=\frac{A-Y}{\alpha}, \]
where $Y$ is a root of \eqref{eq:Y}, because of  the second equation of \eqref{eq:TA}.
\end{remark}

Hence, we have if and only if conditions on the coefficients of $f(x)$ determining its number of roots.

\begin{theorem}\label{th:condt=2n=3}
Let
\[f(x)=-x+b_0x^{q^5}+b_1x^{q^2} \in \mathcal{L}_{6,q}\]
then
\begin{enumerate}
  \item $\dim_{\F_q} \ker f \leq 2$;
  \item $\dim_{\F_q}\ker\,f=2$ if and only if conditions of Theorem \ref{thm:binq6} are satisfied;
  \item $f$ is invertible if and only if
  \begin{small}
        \[ [b_1^{q^2} \left(b_0^{q^4} b_1^{q^3}+b_0 b_1^{q^4}\right)+b_0^{q^2} \left(b_0^{q+1}+b_1^{q^3+q}\right)]^{q^3+1} \neq [b_1^{q^2} \left(b_0^{q^4+q^3}+b_1^{q^4+1}\right)+b_0^{q^2} \left(b_0^{q^3} b_1^q+b_1 b_0^q\right)-1]^{q^3+1}; \]
  \end{small}
  \item $\dim_{\F_q} \ker f=1$ in the remaining cases.
\end{enumerate}
In particular, if $\dim_{\F_q}\ker\,f=2$ then $\N_{q^3/q}(b_0^{q^3+1}-b_1^{q^3+1})=1$.
\end{theorem}

\section{Applications to linear sets}\label{sec:application}

In this section we will explore some possible applications of our results to linear sets.

\smallskip

Let $\Lambda=\PG(W,\F_{q^m})=\PG(1,q^m)$, where $W$ is a vector space of dimension $2$ over $\F_{q^m}$.
A point set $L$ of $\Lambda$ is said to be an \emph{$\F_q$-linear set} of $\Lambda$ of rank $k$ if it is
defined by the non-zero vectors of a $k$-dimensional $\F_q$-vector subspace $U$ of $W$, i.e.
\[L=L_U=\{\la {\bf u} \ra_{\mathbb{F}_{q^m}} \colon {\bf u}\in U\setminus \{{\bf 0} \}\}.\]
We say that two linear sets $L_U$ and $L_W$ of $\Omega=\PG(1,q^m)$ are $\mathrm{P}\Gamma \mathrm{L}$-\emph{equivalent} (or simply \emph{projectively equivalent}) if there exists $\varphi \in \mathrm{P}\Gamma \mathrm{L} (2,q^m)$ such that $\varphi(L_U)=L_W$.

\smallskip

We start by pointing out that if the point $\langle (0,1) \rangle_{\F_{q^m}}$ is not contained in the linear set $L_U$ of rank $m$ of $\PG(1,q^m)$ (which we can always assume after a suitable projectivity), then $U=U_f:=\{(x,f(x))\colon x\in \F_{q^m}\}$ for some $q$-polynomial $\displaystyle f(x)=\sum_{i=0}^{m-1}a_ix^{q^i}\in \tilde{\mathcal{L}}_{m,q}$. In this case we will denote the associated linear set by $L_f$.
Also, recall that the \emph{weight of a point} $P=\langle \mathbf{u} \rangle_{\F_{q^m}}$ is $w_{L_U}(P)=\dim_{\F_q}(U\cap\langle \mathbf{u} \rangle_{\F_{q^m}})$.

\smallskip

Let $x_i$ be the number of points of weight $i$ w.r.t. the linear set $L_U\subseteq \PG(1,q^m)$ of rank $k>0$, then
\begin{equation}\label{eq:rel1}
| L_U |= x_1+\ldots+x_m,
\end{equation}
and
\begin{equation}\label{eq:rel2}
x_1+(q+1)x_2+\ldots+(q^{m-1}+\ldots+q+1)x_m=q^{k-1}+\ldots+q+1,
\end{equation}
see e.g. \cite[Proposition 1.1]{Polverino}.

\subsection{A class of linear sets with small weight spectrum}

Consider the following linear set in $\PG(1,q^{nt})$
\begin{equation}\label{formlinearset}
\mathcal{L}_F:=\{\langle (x,F(x))\rangle_{\F_{q^{nt}}} \colon x \in \F_{q^{nt}} \},
\end{equation}
with
\[F(x)=a_0x^{\sigma}+a_1x^{\sigma q^n}+a_2x^{\sigma q^{2n}}+\ldots+a_{t-1}x^{\sigma q^{n(t-1)}},\]
$\sigma$ a generator of the Galois group $\mathrm{Gal}(\F_{q^n}\colon\F_q)$.
Note that $F(x)=G\circ \sigma$, where $G(x)=a_0x+a_1x^{ q^n}+a_2x^{ q^{2n}}+\ldots+a_{t-1}x^{ q^{n(t-1)}}$.

\begin{theorem}\label{thm:weight}
Let $P$ be a point in $\mathcal{L}_F \subseteq \PG(1,q^{nt})$, then we have that
\begin{itemize}
  \item $1\leq w_{\mathcal{L}_F}(P)\leq t-\dim_{\F_{q^n}} \ker G$, if $P\neq \langle (1,0)\rangle_{\F_{q^{nt}}}$;
  \item $w_{\mathcal{L}_F}(\langle (1,0)\rangle_{\F_{q^{nt}}})=\dim_{\F_q} \ker G=n\cdot \dim_{\F_{q^n}} \ker G$.
\end{itemize}
\end{theorem}
\begin{proof}
Since the point $\langle (0,1) \rangle_{\F_{q^{nt}}}\notin \mathcal{L}_F$, we may assume that $P=\langle(1,m)\rangle_{\F_{q^{nt}}}$ with $m \in \F_{q^{nt}}$.
We have that $w_{\mathcal{L}_F}(\langle (1,m)\rangle_{\F_{q^{nt}}})=i$, for some $m \in \F_{q^{nt}}$, if and only if
\begin{equation}\label{F(x)=mx}
F(x)=mx
\end{equation}
has $q^i$ roots.
If $m=0$, then
\[ w_{\mathcal{L}_F}(\langle (1,0)\rangle_{\F_{q^{nt}}})=\dim_{\F_q} \ker F(x)=\dim_{\F_q} \ker G(x)=nh, \]
with $h=\dim_{\F_{q^n}} \ker G$.
If $m\neq 0$, then we may evaluate the kernel of $f(x)=\frac{1}{m}(-mx+F(x))$, whose dimension will give the value of $w_{\mathcal{L}_F}(\langle (1,m)\rangle_{\F_{q^{nt}}})$. Since $f(x)$ is as in Corollary \ref{cor:Gnoinv}, it follows that
\[ w_{\mathcal{L}_F}(\langle (1,m)\rangle_{\F_{q^{nt}}})=\dim_{\F_q} \ker f(x)\leq t-\dim_{\F_{q^n}} \ker G. \]
\end{proof}

In particular, choices of $G$ having large dimension of the kernel imply that the associated linear set has one point with large weight and the others have small weight. For instance, choosing $G$ as the trace function we get the club defining a particular type of KM-arc.

A \emph{KM-arc of type} $s$ in $\PG(2,q)$ is a set of $q+s$ points of type $(0,2,s)$, i.e. each line of $\PG(2,q)$ meets such a set in either 0, 2 or $s$ points.
The authors in \cite{KM} prove in particular that if a KM-arc of type $s$, with $2<s<q$, in $\PG(2,q)$ exists, then $q$ is even and $s$ is a divisor of $q$.
In \cite{DeBoeckVdV2016}, De Boeck and Van de Voorde established a connection between KM-arcs and $i$-clubs.
An $i$-\emph{club} of rank $m$ in $\PG(1,q^m)$ is an $\F_{q}$-linear set in $\PG(1,q^m)$ such that one point has weight $i$ and all the others have weight one.
The first example of KM-arc presented in \cite{KM} can be described by the following $i$-club, as proved in \cite{DeBoeckVdV2016}:
let $m=nt$, $q=2$, $i=n(t-1)$, $x^\sigma=x^{q^s}$ with $\gcd(s,n)=1$ then the linear set
\begin{equation}\label{i-clubKM}
L_{KM}:=\{\langle (x,L(x)) \rangle_{\F_{2^{nt}}} \colon x \in \F_{2^{nt}}^*\},
\end{equation}
with $L(x)=\mathrm{Tr}_{2^{nt}/2^n}\circ \sigma$, is an $i$-club of $\PG(1,2^{nt})$ defining the example of \cite{KM}, see \cite[Theorem 3.2]{DeBoeckVdV2016}.

\smallskip

Choosing $G(x)=\mathrm{Tr}_{2^{nt}/2^n}$, Theorem \ref{thm:weight} implies again that the linear set \eqref{i-clubKM} is an $i$-club.
In the case in which we choose $G$ such that $\dim_{\F_{q^n}} \ker G =t-2$, setting $F=G\circ\sigma$, Theorem \ref{thm:weight} implies that
\begin{itemize}
  \item $w_{\mathcal{L}_F}(\langle (1,0)\rangle_{\F_{q^{nt}}})=n(t-2)$;
  \item $1\leq w_{\mathcal{L}_F}(P)\leq 2$, for each $P \in \mathcal{L}_F$ and $P\neq \langle (1,0)\rangle_{\F_{q^{nt}}}$.
\end{itemize}
So, this means that in such a case the linear set $\mathcal{L}_F$ is very close to be an $n(t-2)$-club for any choice of $a_0,\ldots,a_{t-1}\in \F_{q^{nt}}$.
It would be of some interest to determine (whether there exist) choices of $a_0,\ldots,a_{t-1}\in \F_{q^{nt}}$ such that $w_{\mathcal{L}_F}(P)<2$ for each point $P\neq \langle (1,0)\rangle_{\F_{q^{nt}}}$, i.e. such that $\mathcal{L}_F$ is an $n(t-2)$-club.

\smallskip

One of the most studied classes of linear sets of the projective line, especially because of their applications (see e.g. \cite{Polverino,Sheekey2016}), is the family of maximum scattered linear sets. A {\it maximum scattered} $\F_q$-linear set of $\PG(1,q^m)$ is an $\F_q$-linear set of rank $m$ of $\PG(1,q^m)$ of size $(q^m-1)/(q-1)$, or equivalently a linear set of rank $m$ in $\PG(1,q^m)$ whose points have weight one.
If $L_f$ is a maximum scattered linear set in $\PG(1,q^m)$, we say also that $f$ is a \emph{scattered polynomial}.
The known scattered polynomials of $\F_{q^m}$ are
\begin{enumerate}
  \item $f_1(x)=x^{q^s}\in \tilde{\mathcal{L}}_{m,q}$, with $\gcd(s,m)=1$, \cite{BL2000};
  \item $f_2(x)=\alpha x^{q^s}+x^{q^{m-s}}\in\tilde{\mathcal{L}}_{m,q}$, with $m\geq 4$, $\gcd(s,m)=1$, $\N_{q^m/q}(\alpha) \notin\{0,1\}$, \cite{LMPT2015,LP2001,Sheekey2016};
  \item $f_3(x)= x^{q^s}+\alpha x^{q^{s+\frac{m}2}}\in\tilde{\mathcal{L}}_{m,q}$, $m \in \{6,8\}$, $\gcd(s,\frac{m}2)=1$ and some conditions on $\alpha$, \cite{CMPZ};
  \item $f_4(x)=x^q+x^{q^3}+\alpha x^{q^5}\in \tilde{\mathcal{L}}_{6,q}$, $q$ odd and $\alpha^2+\alpha=1$, \cite{CsMZ2018,MMZ};
  \item $f_5(x)=h^{q-1}x^q-h^{q^2-1}x^{q^2}+x^{q^4}+x^{q^5}\in \tilde{\mathcal{L}}_{6,q}$, $h\in\F_{q^6}$, $h^{q^3+1}=-1$ and $q$ odd, \cite{BZZ,ZZ}.
\end{enumerate}

Family \eqref{formlinearset} contains most of the known families of maximum scattered linear sets of the line.

\begin{remark}
Let $f_1,f_2,f_3$ and $f_4$ be the polynomials defined above.
\begin{itemize}
  \item Choosing $F(x)=x^\sigma$, with $\sigma$ a generator of $\mathrm{Gal}(\F_{q^{nt}}\colon\F_q)$, we obtain the polynomial $f_1$.
  \item Let $n=2$, $t\geq 1$, $1\leq \ell< 2t$, $\gcd(\ell,2t)=1$ and $x^\sigma=x^{q^{t-\ell}}$. Then $F(x)=\alpha x^\sigma+ x^{\sigma q^{2\ell}}$ coincides with $f_2$ when $m=2t$ and $\N_{q^{2t}/q}(\alpha)\notin\{0,1\}$.
  \item Let $t=2$ and $\sigma$ be a generator of $\mathrm{Gal}(\F_{q^{n}}\colon \F_q)$. Then $F(x)= x^\sigma +\alpha x^{\sigma q^n}$ is, clearly, of type $f_3$.
  \item Let $t=3$, $n=2$ and $x^\sigma=x^q$. Then $F(x)=x^\sigma+x^{q^2 \sigma}+\alpha x^{q^4 \sigma}$, with $\alpha^2+\alpha=1$, coincides with $f_4$.
\end{itemize}
\end{remark}

In \cite[Theorem 7.1]{CMPZ}, the authors prove that for $n=6$ and for each $q>4$ it is possible to find $\alpha \in \F_{q^2}$ such that $f_3$ is a scattered polynomial, without giving the explicit conditions. As a consequence of Theorem \ref{th:condt=2n=3} we are able to determine the if and only if condition on $\alpha$ such that $f_3$ results to be a scattered polynomial.

\begin{theorem}\label{thm:7.3}
The $\F_q$-linear set
\[L_{f_3}=\{\langle (x,x^{q^5}+\alpha x^{q^{2}}) \rangle_{\F_{q^6}} \colon x \in \F_{q^6}^*\}\]
with $\N_{q^6/q^3}(\alpha)\neq0,1$,
is scattered if and only if  the equation
\begin{equation}\label{eq:YLS}
Y^2-(\mathrm{Tr}_{q^3/q}(A)-1)Y+\mathrm{N}_{q^3/q}(A)=0,
\end{equation}
with $\displaystyle A=\frac{-\alpha^{q^3+1}}{1-\alpha^{q^3+1}}$, admits two roots over $\F_q$(\footnote{Denoting by $\beta=-\mathrm{Tr}_{q^3/q}(A)+1$ and $\gamma=\mathrm{N}_{q^3/q}(A)$, this happens when $q$ is odd and $\beta^2-4\gamma$ is a square over $\F_q$ or when $q$ is even and $\tr_{q/2}(\gamma/\beta^2)=0$.}).
In particular, there always exists such a $\alpha$ for any $q>2$.
\end{theorem}
\begin{proof}
The linear set $L_{f_3}$ is scattered if and only if $w_{L_{f_3}}(\langle (1,m)\rangle_{\F_{q^6}})\leq 1$ for each $m \in \F_{q^6}^*$, since the point $\langle (1,0)\rangle_{\F_{q^6}}\notin L_{f_3}$. This is equivalent to require that for each $m \in \F_{q^6}^*$
\[ \dim_{\F_q} \ker \left( -x+\frac{1}m x^{q^5}+\frac{\alpha }m x^{q^2} \right) \leq 1. \]
By Theorem \ref{th:condt=2n=3}, it follows that $\dim_{\F_q} \ker \left( -x+\frac{1}m x^{q^5}+\frac{\alpha }m x^{q^2} \right) \leq 2$ and clearly, if $A$ is as in the statement, by Theorem \ref{thm:binq6} we have that
\[\dim_{\F_q} \ker \left( -x+\frac{1}m x^{q^5}+\frac{\alpha}m x^{q^2} \right) \leq 1\]
and hence $L_{f_3}$ is scattered.
Now, suppose that \eqref{eq:YLS} admits either one root over $\F_q$ or two roots in $\F_{q^2}\setminus\F_q$, by the second part of Theorem \ref{thm:binq6} there exists $m\in \F_{q^6}^*$ such that
\[\dim_{\F_q} \ker \left( -x+\frac{1}m x^{q^5}+\frac{\alpha }m x^{q^2} \right) =2,\]
i.e. $w_{L_{f_3}}(\langle (1,m)\rangle_{\F_{q^6}})=2$ proving that $L_{f_3}$ is not scattered.
For the second part, let $Y^2+aY+b =0$ any equation over $\F_{q}$ admitting two roots over $\F_q$ with $b\neq 0$, then $q>2$ and by \cite{Katz} (see also \cite[Theorems 1.1 and 1.2]{Moisio}), there exists $A\in \F_{q^3}$ such that $\mathrm{Tr}_{q^3/q}(A)=-a+1$, $\mathrm{N}_{q^3/q}(A)=b$ and $A\neq1$.
Then, by the first part, for each $\alpha\in \F_{q^6}^*$ such that
\[ \N_{q^6/q^3}(\alpha)=\frac{A}{A-1}, \]
the linear set
\[L_{f_3}=\{\langle (x,x^{q^5}+\alpha x^{q^{2}}) \rangle_{\F_{q^6}} \colon x \in \F_{q^6}^*\}\]
is scattered.
\end{proof}

In \cite[Corollary 5.4]{CMPZ} the authors prove that the number of points of $L_{f_3}$ with weight two is a multiple of $q^2+q+1$. As a consequence of Theorem \ref{thm:binq6}, we can completely determine the number of such points and the cardinality of $L_{f_3}$.

\begin{corollary}
The $\F_q$-linear set
\[L_{f_3}=\{\langle (x,x^{q^5}+\alpha x^{q^{2}}) \rangle_{\F_{q^6}} \colon x \in \F_{q^6}^*\}\]
with $\N_{q^6/q^3}(\alpha)\neq0,1$, has $x_2$ points of weight two, where
\[ x_2=\left\{ \begin{array}{lll} 2(q^2+q+1) & \text{if}\quad \eqref{eq:YLS} \quad \text{has two roots over}\quad \F_{q^2}\setminus\F_q \\ q^2+q+1 & \text{if}\quad \eqref{eq:YLS} \quad \text{has one root over}\quad \F_q \\
0 & \text{otherwise}  \end{array} \right.. \]
In particular,
\[ |L_{f_3}|=\left\{ \begin{array}{lll} q^5+q^4-q^3-q^2-q+1 & \text{if}\quad \eqref{eq:YLS} \quad \text{has two roots over}\quad \F_{q^2}\setminus\F_q \\ q^5+q^4+1 & \text{if}\quad \eqref{eq:YLS} \quad \text{has one root over}\quad \F_q \\
\frac{q^6-1}{q-1} & \text{otherwise}  \end{array} \right..  \]
\end{corollary}
\begin{proof}
The assertion follows by the previous result and from the last part of Theorem \ref{thm:binq6}. Indeed, the number of points with weight two corresponds to the number of $m \in \F_{q^6}^*$ such that
\[ \dim_{\F_q} \ker \left( -x+\frac{1}m x^{q^5}+\frac{\alpha}m x^{q^2} \right) =2, \]
i.e. with the number of solutions of
\[ x^{q^2+q+1}=\frac{A-Y}\alpha, \]
where $Y$ is a solution of \eqref{eq:YLS}.
The last part follows by the following relations
\[ |L_{f_3}|=x_1+x_2, \]
and
\[ x_1+(q+1)x_2=\frac{q^6-1}{q-1}, \]
where $x_1$ is the number of points having weight one w.r.t. $L_{f_3}$.
\end{proof}

\begin{remark}
In \cite{BCsM}, Bartoli, Csajb\'ok and Montanucci, independently and with different techniques, characterize scattered linear sets of shape $L_{f_3}$ obtaining the same conditions of Theorem \ref{thm:7.3}.
Also, they use such conditions to prove a conjecture posed in \cite{CMPZ} on the number of new maximum scattered subspaces defining linear sets of type $L_{f_3}$.
Whereas, in \cite{PZZ} the authors prove that linear sets of shape \eqref{formlinearset} obtained with $t=2$ are not scattered when $n$ is large enough.
\end{remark}

\bigskip

\noindent Olga Polverino and Ferdinando Zullo\\
Dipartimento di Matematica e Fisica,\\
Universit\`a degli Studi della Campania ``Luigi Vanvitelli'',\\
I--\,81100 Caserta, Italy\\
{{\em \{olga.polverino,ferdinando.zullo\}@unicampania.it}}

\end{document}